\documentclass[12pt,a4paper]{article}

\usepackage[naustrian, english]{babel} 
\usepackage[T1]{fontenc}
\usepackage{lmodern}
\usepackage{textcomp}
\usepackage{amssymb} 
\usepackage{amsthm}
\usepackage{mathtools}
\usepackage[pdftex]{graphicx}
\usepackage[export]{adjustbox}
\usepackage{caption}
\usepackage{color}
\usepackage[arrow, matrix, curve]{xy}
\usepackage{lmodern}
\usepackage[utf8]{inputenc}
\usepackage{CJKutf8}
\usepackage{url}
\usepackage{hyperref}
\usepackage{float}
\usepackage{listings}
\usepackage{appendix}
\usepackage{pdfpages}
\usepackage{algorithm}
\usepackage{algorithmic}
\usepackage{geometry}
\usepackage{bbm}
\usepackage{bbold}
\usepackage{graphicx}
\usepackage{caption}
\usepackage{subcaption}
\usepackage{xargs} % use more than one optional parameter in a new command
\usepackage[colorinlistoftodos,textsize=small]{todonotes}
\usepackage{soul}

\newcommand{\R}{\mathbb{R}}

\newcommand{\pa}{\partial}

\newcommand{\ve}{\varepsilon}

\newcommand{\vp}{\varphi}

\newcommand{\md}{\mathrm{d}}

\newcommand{\vpa}{\varphi_\ast}

\newcommand{\fa}{f_\ast}
\newcommand{\feps}{f_\ve}
\newcommand{\geps}{g_\ve}
\newcommand{\tf}{\tilde{f}}

\newcommand{\af}{f_*}

\newcommand{\sgn}{\operatorname{sgn}}

\newcommand{\be}[1]{\begin{equation}\label{#1}}
\newcommand{\ee}{\end{equation}}
\newcommandx{\unsure}[2][1=]{\todo[linecolor=red,backgroundcolor=red!25,bordercolor=red,#1]{#2}}
\newcommandx{\change}[2][1=]{\todo[linecolor=blue,backgroundcolor=blue!25,bordercolor=blue,#1]{#2}}
\newcommandx{\info}[2][1=]{\todo[linecolor=green,backgroundcolor=green!25,bordercolor=green,#1]{#2}}
\newcommandx{\improvement}[2][1=]{\todo[linecolor=yellow,backgroundcolor=yellow!25,bordercolor=yellow,#1]{#2}}

\newcommandx{\biblio}[2][1=]{\todo[linecolor=blue,backgroundcolor=magenta!25,bordercolor=blue,#1]{#2}}
\newcommandx{\laura}[2][1=]{\todo[linecolor=violet,backgroundcolor=violet!25,bordercolor=violet,#1]{#2}}

\newtheorem{theorem}{Theorem}
\newtheorem{lemma}[theorem]{Lemma}

\usepackage{tikz}
\usetikzlibrary{positioning}
\tikzset{>=stealth}

\begin{document}
	
\title{Two kinetic models for non-instantaneous binary alignment collisions}
\author{L. Kanzler\thanks{CEREMADE, Université Paris Dauphine, Place du Maréchal de Lattre de
Tassigny, F-75775 Paris Cedex 16, France.\tt{laura.kanzler@dauphine.psl.eu}} \and 
C. Schmeiser\thanks{University of Vienna, Faculty for Mathematics, Oskar-Morgenstern-Platz 1, 1090 Wien, Austria. 
{\tt christian.schmeiser@univie.ac.at}} \and 
V. Tora \thanks{University of Rome "Tor Vergata", Department of Mathematics, Via della Ricerca Scientifica, 1, 00133, Rome, Italy.
	{\tt  veronica.tora2@unibo.it}} \and }
	\date{\vspace{-5ex}}
	\maketitle
	
\begin{abstract}
A new type of kinetic models with non-instantaneous binary collisions is considered. Collisions are described by a transport process in the joint state space of a pair of particles.
The interactions are of alignment type, where the states of the particles approach each other. For two spatially homogeneous models with deterministic or stochastic collision
times existence and uniqueness of solutions, the long time behavior, and the instantaneous limit are considered, where the latter leads to standard kinetic models of 
Boltzmann type.
\end{abstract}

\medskip


{\bf\em Acknowledgments:} This work has been supported by the Austrian Science Fund (grant nos. F65 and W1245). C.S. also acknowledges the hospitality 
of the Institut Henri Poincar\'e (UAR 839 CNRS-Sorbonne Universit\'e), and LabEx CARMIN (ANR-10-LABX-59-01).

\section{Introduction}

The main aim of this work is to initiate the investigation of a new class of kinetic models for ensembles of particles undergoing binary, \emph{non-instantaneous} collisions. 
The idea is to replace instantaneous jumps in the joint state space of a pair of particles by continuous processes taking finite time. The duration of the collision process can be deterministic (as, for example, the interaction of two soft elastic balls) or stochastic. 
Models of this new type have so far not appeared in the mathematical literature, but related models can be found in the physics literature dealing with non-instantaneous interactions of quantum particles (see, e.g. \cite{c5:LSM}).

The standard approach of kinetic theory is to model interactions between agents via \emph{jump processes} on the state space, with the Boltzmann equation of gas dynamics \cite{c5:CIP} as prototypical example. This is an idealization in the sense that in reality these interactions take a finite time span, where states change in a continuous fashion. For passive
particles, such as gas molecules, the approximation by instantaneous collisions is typically consistent with the limit of small particle size. The present work is motivated by 
attempts to model ensembles of living agents, where the changes of state are the result of often complicated internal processes, and not simple mechanical interactions.
Examples are the run-and-tumble behavior of E. coli bacteria, where the typical instantaneous modeling \cite{CalRaoSch} of the tumble phase is somewhat questionable,
since its actual duration is about 10\% of the duration of the run phase \cite{Berg}. Other examples with possibly long interaction times are contact inhibition of movement
upon cell-cell collisions \cite{Guelstein} or collisions under the presence of cell-cell adhesion \cite{Harvey}.

As a first step in the mathematical treatment of non-instantaneous collisions, we shall analyze two models, which can be interpreted as simple descriptions of alignment or opinion formation processes, where in the latter case collisions represent
discussion processes between two individuals leading towards convergence of opinions. 

In the {\bf Stochastic Collision Time Model} (SCTM) we assume the duration of the process to be stochastic and governed by a Poisson process with constant parameter $\gamma>0$. Each particle can be involved in a binary collision with a second particle, or it can be in a free state, i.e. between collisions.
The model consists of two coupled equations, one for the distribution function $f=f(\vp,t)$ at time $t>0$ of free particles with respect to the state $\vp\in\R$, and the other 
for the distribution function $g=g(\vp,\vpa,t)$ of pairs of particles in a collision process with $\vp, \vpa \in \R$. The system is of the form
\begin{equation}\label{mainmodel}
\begin{split}
	&\pa_t f = 2\left(\gamma \int_{\R} g \: \md \vpa- \lambda f \int_{\R} \fa\: \md \vpa  \right) \,, \\
	&\pa_t g + \nabla \cdot (v_1 g)=\lambda \fa f -  \gamma g \,,
\end{split}
\end{equation}
where the abbreviation $\fa := f(\vpa,t)$ is used. The \emph{collision rate} with rate constant $\lambda>0$ is assumed to be independent from the pre-collisional states. 
The factor 2 in the first equation is due
to the fact that pairs of free flying particles are lost/gained at the beginning/end of a collision. The transport term with $\nabla = (\partial_\vp,\partial_{\vpa})$ in the second equation describes the collision process
\begin{equation}\label{ODE-charac}
  \begin{pmatrix} \dot\vp \\ \dot\vpa \end{pmatrix} = v_1(\vp,\vpa) := \frac{\vp-\vpa}{2} \begin{pmatrix} -1 \\ 1 \end{pmatrix} \,.
\end{equation}
We shall be interested in the initial value problem with
\begin{equation}\label{IC}
  f(\vp,0) = f_0(\vp) \,,\quad g(\vp,\vpa,0) = g_0(\vp,\vpa) \,,\qquad \vp,\vpa \in \R \,,
\end{equation}
where the initial data satisfy 
\begin{equation}\label{IC-ass}
    f_0,g_0 \ge 0 \,,\quad \int_\R (1+\vp^2) f_0 \:\md\vp < \infty \,,\quad \int_{\R^2} (1+ \vp^2) g_0 \:\md\vp\:\md\vpa < \infty \,,\quad g_0(\vp,\vpa) = g_0(\vpa,\vp)\,.
\end{equation}
Note that the collision dynamics propagates the {\em indistinguishability} property, i.e. 
$$
   g(\vp,\vpa,t) = g(\vpa,\vp,t) \,,\qquad \mbox{for } \vp,\vpa\in\R \,,\quad t>0 \,,
$$ 
holds for (unique) solutions of the initial value problem.

In the {\bf Deterministic Collision Time Model} (DCTM) the collision process is governed by an ODE system like \eqref{ODE-charac}, but with the drift velocity $$v_2(\vp,\vpa) = \frac{\sgn(\vp-\vpa)}{2}\begin{pmatrix} -1 \\1
\end{pmatrix}\,.$$
In this case complete alignment, i.e. $\vp=\vpa$, is reached in finite time, after which the collision ends. This leads to the model
\begin{equation}\label{mainmodel-d}
\begin{split}
	&\pa_t f = 2\left(2\bar g- \lambda f \int_{\R} \fa\: \md \vpa  \right) \,, \qquad \bar g(\vp,t) = g(\vp,\vp,t) \,,\\
	&\pa_t g + \nabla \cdot (v_2 g)=\lambda \fa f  \,,\qquad \vp\ne\vpa \,.
\end{split}
\end{equation}
The term $2\bar g$ should be interpreted as the sum of the traces of $g$, as the main diagonal in the $(\vp,\vpa)$-plane is approached from $\vp>\vpa$ and $\vp<\vpa$. 
The traces are equal by the indistinguishability property.

More general models can be produced in various ways. For example 
\begin{itemize}
\item Other choices of collision potentials can be used instead of $v_1$, $v_2$.
\item The collision rate $\lambda$ could depend on the pre-collisional states $\vp$ and $\vpa$.
\item The collision stopping rate $\gamma$ in the SCTM could depend on the two-particle-state $(\vp,\vpa)$.  
\item In the DCTM the interaction potential $v_1$ could be combined with a fixed prescribed collision time
    or with a collision time, where the pre-collisional distance is reduced by a fixed ratio.
\end{itemize}

Explicit solution of the equations for $g$ in \eqref{mainmodel} and in \eqref{mainmodel-d} (actually carried out below in Sections \ref{c5:ex_prob} and, respectively, \ref{sec:det_ex}) and substitution
into the corresponding $f$-equations leads to kinetic models with time delays in the gain terms, reminiscent of formally derived semiclassical models for interacting
quantum particles \cite{c5:LSM}. It has been observed that these models violate the conservation laws of mass and energy, which is a straightforward observation for the
models considered here, when only the free particles with distribution function $f$ are considered. In the terminology of \cite{c5:LSM}, the pair distribution $g$ in our models
is an explicit account of the so called \emph{correlated density,} used for correcting the conservation laws.

There is a sizeable literature on models of alignment interactions, falling into two classes. Swarming models describing flocking behavior are typically based on mean-field
interactions with the Vicsek model \cite{c5:VCBCS} as a classical example and with, e.g.  \cite{c5:BCC,c5:degond}, as kinetic versions. 
On the other hand, alignment of rod-shaped polymers or of bacterial colonies is driven by pairwise interactions and has been modeled as instantaneous binary collisions 
\cite{c5:aronson,c5:bertin,c5:carlen,c5:HKMS}. Similar models are used for describing opinion formation \cite{Toscani} and the exchange of proteins between cells \cite{c5:HFMW}. Typical results are long-time convergence to aligned equilibria as well as the derivation of macroscopic models for aligned states
\cite{c5:DFR,c5:degond,c5:HKMS}. These models share the property of shrinking phase space volume with the inelastic Boltzmann equation for granular gases
 \cite{c5:bobylev1, c5:carillo, c5:mischler}. 
 
In the setting of gas dynamics with instantaneous collisions, the restriction to binary interactions is justified \cite{c5:CIP}. In the case of non-instantaneous collisions it would be incorrect since, strictly speaking, there is a positive probability that during an extended collision period a colliding pair is joined by additional particles. We expect, however, that for short collision times the importance of higher order collisions decreases with their order. A further investigation of this question will be the subject of future work. Extended models with more than two collision partners might be a fruitful approach to the problem of kinetic modeling of higher order chemical reactions. 

The SCTM has the special property that the moments of $f$ and $g$ with a certain order satisfy closed systems of ODEs. This property is presented in the following section
and used to derive not only conservation of the total mass and of the mean state, but also the long-time distribution of the mass between the free particles and those involved 
in collisions, as well as the fact that the variances decay to zero. The latter is the essential information used in the proof of decay to an aligned state, carried out in Section \ref{c5:ex_prob}
together with an existence and uniqueness proof for the SCTM. Section \ref{sec:il1} is dedicated to the instantaneous limit in the SCTM, where a rescaling is introduced making the
collision process fast and short. Some formal properties of the DCTM are collected in Section \ref{c5:noninstdet}, and an existence and uniqueness result is proved in Section \ref{sec:det_ex}. It is weaker
than for the SCTM in the sense that no continuity in time is obtained because of a lack of control of the trace $\bar g$. This also causes the absence of a decay-to-equilibrium
result for the DCTM and the lack of a rigorous justification of the instantaneous limit, formally carried out in Section \ref{sec:il2}.

\section{Formal properties of the Stochastic Collision Time Model -- moments}\label{sec:formal}

\paragraph{Collision rules:} Integration of the characteristic equations \eqref{ODE-charac} with initial state $(\vp',\vpa')$ for the duration $s$ of a collision gives the 
collision rules, i.e. the map from pre-collisional states $(\vp',\vpa')$ to post-collisional states $(\vp,\vpa)$,
\begin{equation}\label{coll-rule}\begin{split}
	\vp &= \Phi^s(\vp',\vpa') := \frac{\vp'+\vpa'}{2}+e^{-s}\frac{\vp'-\vpa'}{2} \,, \\
	\vpa &= \Phi_*^s(\vp',\vpa') := \frac{\vp'+\vpa'}{2}+e^{-s}\frac{\vpa'-\vp'}{2} \,,
\end{split}\end{equation}
which we would also use in a kinetic model with instantaneous collisions. Note that the pre-collisional states can be computed from the post-collisional ones by time-inversion:
$\vp' = \Phi^{-s}(\vp,\vpa)$, $\vpa' = \Phi_*^{-s}(\vp,\vpa)$.

\paragraph{Moments:} We expect the total mass
\begin{align}\label{masses}
   M := M_f + 2M_g \qquad\mbox{with } \quad M_f:= \int_{\R} f \: \md\vp \quad\text{and}\quad M_g := \int_{\R^2} g \: \md\vp\:\md\vpa 
\end{align}
to be conserved, but we can actually obtain more detailed information, since the partial masses solve the closed ODE system
\begin{equation}\label{ODEmasses}
	\begin{split}
		\dot{M_f} &= 2\gamma M_g - 2\lambda M_f^2 \,,\\
		\dot{M_g} &= \lambda M_f^2 - \gamma M_g \,.
	\end{split}
\end{equation}
Noting the desired conservation property
$$
   M = M_{f_0} + 2 M_{g_0} \,,
$$
it can be solved explicitly, establishing exponential convergence of $(M_f(t),M_g(t))$ as $t\to\infty$ to
\begin{equation}\label{M-infty}
   (M_{f_\infty},\, M_{g_\infty}) := 
    \frac{2M}{1 + \sqrt{1 + 8\lambda M/\gamma}} \left( 1,\, \frac{2\lambda M/\gamma}{1 + \sqrt{1 + 8\lambda M/\gamma}}\right) \,.
\end{equation}
Analogously, we obtain for the first order moments
\begin{align}\label{momenta}
   I := I_f + 2I_g \qquad\mbox{with } \quad I_f:= \int_{\R} \vp f \: \md\vp \quad\text{and}\quad I_g := \int_{\R^2} \vp g \: \md\vp\:\md\vpa 
\end{align}
the ODE system
\begin{equation}\label{ODEmomenta}
	\begin{split}
		\dot{I_f} &= 2\gamma I_g - 2\lambda I_f M_f \,,\\
		\dot{I_g} &= \lambda I_f M_f - \gamma I_g \,,
	\end{split}
\end{equation}
which can again be solved explicitly, leading to the second conservation law
$$
   I = I_{f_0} + 2 I_{g_0} \,,
$$
and to the convergence of $(I_f(t),I_g(t))$ as $t\to\infty$ to
\begin{align}\label{limits-momenta}
	(I_{f_\infty},\, I_{g_\infty}) = \frac{2I}{1 + \sqrt{1 + 8\lambda M/\gamma}} \left( 1,\, \frac{2\lambda M/\gamma}{1 + \sqrt{1 + 8\lambda M/\gamma}}\right) \,.
\end{align}
With the mean state
\begin{equation}\label{phiinfty}
\vp_\infty :=\frac{I}{M}=\frac{I_{f_\infty}}{M_{f_\infty}}=\frac{I_{g_\infty}}{M_{g_\infty}}  \,,
\end{equation}
we define the  variances
\begin{align*}
	V_f:=\int_{\R} (\vp-\vp_\infty)^2 f \, \md \vp \,, \qquad V_g:=\int_{\R} (\vp-\vp_\infty)^2 g \, \md \vpa \md \vp \,,
\end{align*}
as well as the additional variance-type second order moment
$$
    \bar V_g := \int_{\R} (\vp-\vpa)^2 g \, \md \vpa \md \vp \,.
$$
These three quantities satisfy the ODE system
\begin{equation}\label{ODEvariances}
	\begin{split}
		\dot{V_f} &= 2\gamma V_g - 2\lambda M_f V_f \,,\\
		\dot{V_g} &= \lambda M_f V_f - \gamma V_g - \bar V_g \,,\\
		\dot{\bar V}_g &= 2\lambda M_f V_f - (2+\gamma)\bar V_g + 2\lambda(I_f - \vp_\infty M_f)^2 \,.
	\end{split}
\end{equation}
This shows that the total variance $V:= V_f+2V_g$ is nonincreasing. However, investigation of the full ODE system provides a much stronger result.

\begin{lemma}
Let $\lambda,\gamma>0$, let the assumptions \eqref{IC-ass} on the initial data hold with $M>0$, and let $M_f$ and $I_f$ be given by solving \eqref{ODEmasses}, 
\eqref{ODEmomenta}. Then there exist constants $C,\mu>0$ such that the solution $(V_f,V_g,\bar V_g)$ of \eqref{ODEvariances} satisfies
$$
    V_f(t) + V_g(t) + \bar V_g(t) \le C e^{-\mu t} \,,\qquad t\ge 0 \,.
$$
\end{lemma}

\begin{proof}
From the explicit solutions of \eqref{ODEmasses}, \eqref{ODEmomenta} we deduce exponential convergence of the coefficient matrix and of the inhomogeneity 
of the linear system \eqref{ODEvariances}. The limit of the inhomogeneity vanishes by \eqref{phiinfty}, and the limit of the coefficient matrix can be shown to 
have eigenvalues with negative real parts by the Routh-Hurwitz criterion \cite{Hurwitz} (Here it is used that $M_{f_\infty} > 0$ by $M>0$).
The lemma then follows from standard results for ODE systems.
\end{proof}

\paragraph{Equilibria:}
By the decay of $V_f$ we expect $f$ to converge to a Delta-distribution as $t\to\infty$. The same is true for the one-particle marginal $\int_{\R}g\,\md\vpa$ of $g$.
The observation 
$$
    \int_{\R^2} (\vp-\vp_\infty)(\vpa-\vp_\infty) g\,\md\vpa\md\vp = V_g - \frac{1}{2} \bar V_g
$$
shows that the correlation between particle pairs in collisions tends to zero. Therefore we expect convergence to the equilibrium distributions
\begin{align}\label{globeq}
  f_{\infty}(\vp) =  M_{f_\infty}\delta(\vp-\vp_{\infty}) \,,\qquad g_{\infty}(\vp,\vpa) = M_{g_\infty} \delta(\vp-\vp_{\infty})\delta(\vpa-\vp_{\infty}) \,.
\end{align}

\paragraph{Entropy:} 
We introduce an entropy functional adapted to the exchange terms between collisional and non-collisional states:
\begin{equation}\label{entr}
     \mathcal{H}[f,g] := \int_{\R} f(\log(\lambda f) - 1)\md\vp + \int_{\R^2} g \left( \log(\gamma g) - 1 \right)\md\vpa\md\vp \,,
\end{equation}
whose time derivative along solutions of \eqref{mainmodel} is given by
\begin{equation}\label{entr-diss}
   \frac{\md}{\md t} \mathcal{H}[f,g] = -\int_{\R^2} (\lambda f\fa - \gamma g)\log\left( \frac{\lambda f\fa}{\gamma g}\right)\md\vpa\md\vp + M_g \,.
\end{equation}
Because of the appearance of the positive term coming from the drift in the $g$-equation, this is not useful for the analysis of the long-time behaviour.
However, for finite times it provides a $(L\log L)$-bound, which will be convenient in the analysis of the instantaneous limit.

\section{Existence, uniqueness, and convergence to equilibrium for the Stochastic Collision Time Model}\label{c5:ex_prob}

\paragraph{Global existence and uniqueness:} With the semigroup
\begin{equation}\label{semigroup}
   (S(t)g)(\vp,\vpa) = 
    e^{(1-\gamma)t} g\left( \Phi^{-t}(\vp,\vpa), \Phi_*^{-t}(\vp,\vpa)\right) \,,
\end{equation}
generated by the operator $Gg = -\nabla\cdot(v_1 g)-\gamma g = -v_1\cdot\nabla g +(1-\gamma)g$, and with
$$
    F(s,t) = \exp\left( -2\lambda \int_s^t M_f(s)ds\right) \,,
$$
we obtain the mild formulation
\begin{equation}\label{mild}
	\begin{split}
		&f(\vp,t)  = F(0,t) f_0(\vp) + 2\gamma\int_0^t F(s,t)\int_{\R} g(\vp,\vpa,s) \:\md\vpa \:\md s  \,,\\
		&g(\vp,\vpa,t) = (S(t)g_0)(\vp,\vpa) + \lambda \int_0^t (S(t-s)f(\cdot,s)\fa(\cdot,s))(\vp,\vpa) \: \md s \,,
	\end{split}
\end{equation}
of the initial value problem \eqref{mainmodel}, \eqref{IC}. After having solved problem \eqref{ODEmasses}, we may consider $M_f$ and therefore also 
$F$ as given.

\begin{theorem}\label{thm:ex_un}
	Let $f_0 \in L_+^1(\R)$, $g_0 \in L_+^1(\R^2)$. Then \eqref{mild} has a unique solution $(f,g) \in C\big([0,\infty);\: L_+^1(\R)\times L_+^1(\R^2)\big)$.
\end{theorem}
\begin{proof}
Obviously, Picard iteration for \eqref{mild} preserves nonnegativity. For proving the contraction property, we use
\begin{eqnarray*}
  \int_{\R} \left| \int_0^t F(s,t)\int_{\R} (g(\vp,\vpa,s)-\tilde g(\vp,\vpa,s))\md\vpa \:\md s \right| \md\vp
  \le t \sup_{0<s<t} \|g(\cdot,\cdot,s) - \tilde g(\cdot,\cdot,s)\|_{L^1(\R^2)} 
\end{eqnarray*}
and
\begin{eqnarray*}
  &&\int_{\R^2} \left| \int_0^t S(t-s)(f\fa - \tilde f \tilde f_*)(\vp,\vpa) \: \md s \right| \md\vpa\md\vp \\
  &&\le \int_0^t e^{t-s} \int_{\R^2} \left| f(\Phi^{s-t}(\vp,\vpa),s) f(\Phi^{s-t}_*(\vp,\vpa),s) - \tilde f(\Phi^{s-t}(\vp,\vpa),s) \tilde f(\Phi^{s-t}_*(\vp,\vpa),s)\right|
             \md\vpa\md\vp\md s \\
  &&= \int_0^t  \left\| f(\cdot,s)\fa(\cdot,s) - \tilde f(\cdot,s) \tilde f_*(\cdot,s)\right\|_{L^1(\R^2)} \md s 
       \le 2t \sup_{s>0} M_f(s)  \sup_{0<s<t} \left\| f(\cdot,s)-\tilde f(\cdot,s)\right\|_{L^1(\R)} \,.
\end{eqnarray*}
For the first inequality we have used $e^{-\gamma(t-s)}\le 1$. The equation afterwards is due to the coordinate change 
$(\Phi^{s-t}(\vp,\vpa),\Phi^{s-t}_*(\vp,\vpa)) \to (\vp,\vpa)$. For the last inequality we have used $f\fa - \tilde f \tilde f_* = (f-\tilde f)\fa + \tilde f (\fa - \tilde f_*)$.

Since continuity with respect to time is obvious, we obtain local existence. The convergence of $(M_f(t),M_g(t))$ as $t\to\infty$ implies a global $L^1$-bound and therefore
global existence.
\end{proof}

\paragraph{Weak convergence to equilibrium:}

With the results on the moments the following convergence result is easily proved.

\begin{theorem}\label{c5:decay}
Let the initial data $(f_0,g_0)$ satisfy \eqref{IC-ass} (and therefore the assumptions of Theorem \ref{thm:ex_un}). Then the mild solution $(f(t),g(t))$ of
\eqref{mainmodel}, \eqref{IC} satisfies 
	\begin{align}\label{c5:weakconvergence}
		\lim_{t\to\infty}(f(t),g(t)) = \left(M_{f_{\infty}}\delta_{\vp_{\infty}},M_{g_{\infty}}\delta_{(\vp_{\infty},\vp_{\infty})}\right) \,,
	\end{align}
	in the sense of distributions, where  $M_{f_{\infty}}$, $M_{g_{\infty}}$, and $\vp_\infty$ are given by \eqref{M-infty}, \eqref{phiinfty}. 
\end{theorem}
\begin{proof}
For a test function $h_1 \in C^1_b(\R)$ we have
\begin{eqnarray*}
  && \left| \int_{\R} f(\vp,t)h_1(\vp)\md\vp - M_{f_\infty} h_1(\vp_\infty)\right| \\
 && =\left| \int_{\R} (f(\vp,t)(h_1(\vp)-h_1(\vp_\infty)) +  h_1(\vp_\infty)(f(\vp,t)- f_\infty(\vp)))\md\vp \right| \\ 
  &&\le \|h_1'\|_{L^\infty(\R)} \int_{\R} f(\vp,t)|\vp-\vp_\infty| \md\vp + \|h_1\|_{L^\infty(\R)} |M_f(t) - M_{f_\infty}| \\
  &&\le \|h_1'\|_{L^\infty(\R)}\sqrt{M_f(t) V_f(t)} + \|h_1\|_{L^\infty(\R)} |M_f(t) - M_{f_\infty}|  \,,
\end{eqnarray*}
where the Cauchy-Schwarz inequality has been used for the second estimate. By the results of the previous section, this completes the proof of the convergence of $f(t)$.
Analogously, for $h_2 \in C^1_b(\R^2)$ we have
\begin{eqnarray*}
  && \left| \int_{\R^2} g(\vp,\vpa,t)h_2(\vp,\vpa)\md\vpa\md\vp - M_{g_\infty} h_2(\vp_\infty,\vp_\infty)\right| \\
  &&\le \|\nabla h_2\|_{L^\infty(\R^2)} \int_{\R^2} g(\vp,\vpa,t)(|\vp-\vp_\infty| + |\vpa-\vp_\infty| )\md\vpa\md\vp + \|h_2\|_{L^\infty(\R^2)} |M_g(t) - M_{g_\infty}| \\
  &&\le 2 \|\nabla h_2\|_{L^\infty(\R^2)}\sqrt{M_g(t) V_g(t)} + \|h_2\|_{L^\infty(\R^2)} |M_g(t) - M_{g_\infty}|  \,,
\end{eqnarray*}
completing the proof.
\end{proof}

\section{The instantaneous limit for the Stochastic Collision Time Model}\label{sec:il1}

\paragraph{The formal limit:} Collisions are close to instantanteous if they proceed fast and last a short time. In this situation we expect the number of particle pairs involved in collisions to be
small. These observations motivate the rescaling
\begin{equation}\label{scaling}
   v_1 \to \frac{v_1}{\ve} \,,\quad \gamma \to \frac{\gamma}{\ve} \,,\quad g \to \ve g \,,
\end{equation}
with a small positive parameter $\ve$.
This results in the singularly perturbed system
\begin{equation}\label{model_eps}
	\begin{split}
		&\pa_t \feps = 2\Big(\gamma\int_{\R} \geps \: \md \vpa- \lambda M_{\feps} \feps \Big) \,, \\
		&\ve \pa_t \geps + \nabla \cdot(v_1 \geps) = \lambda \feps f_{\ve,\ast} -  \gamma \geps \,.
	\end{split}
\end{equation} 
We shall assume initial conditions respecting the rescaling in the sense that for the rescaled variables $(\feps,\geps)$ we still pose the initial conditions \eqref{IC} with 
$\ve$-independent initial data satisfying \eqref{IC-ass}. The formal limit $\ve \to 0$ 
\begin{equation}\label{c5:inst_model}
	\begin{split}
		&\pa_t f = 2\Big(\gamma \int_{\R} g \: \md \vpa- \lambda M_f f \Big), \\
		&\nabla \cdot(v_1 g) = \lambda f\af -  \gamma g,
	\end{split}
\end{equation} 
involves a quasi-stationary equation for $g$, which can be solved by  passing to the limit $\ve\to 0$ in the rescaled mild formulation \eqref{mild},
$$
  \geps(\vp,\vpa,t) = (S(t/\ve)g_0)(\vp,\vpa) +\lambda \int_0^{t/\ve} (S(s)f(\cdot,t-\ve s)\fa(\cdot,t-\ve s))(\vp,\vpa) \: \md s \,,
$$
giving
\begin{equation}\label{mild-limit}
    g(\cdot,\cdot,t) = \lambda \int_0^\infty S(s)f(\cdot,t)\fa(\cdot,t) \:\md s = \lambda \int_0^\infty e^{(1-\gamma)s} f(\Phi^{-s},t) f(\Phi^{-s}_*,t)\md s \,.
\end{equation}
We recall that $(\Phi^{-s}(\vp,\vpa),\Phi^{-s}_*(\vp,\vpa))$ is the pre-collisional state 
corresponding to the post-collisional state $(\vp,\vpa)$ after a collision of duration $s$. Finally, by substitution of \eqref{mild-limit} into the $f$-equation in \eqref{c5:inst_model}, we can write the limiting equation for $f$ in the standard kinetic form
\begin{equation}\label{f-kin}
   \pa_t f = Q_1(f,f) := \int_{\R} \int_0^\infty \sigma(s)(e^s f'f_*' - f\fa) \md s\:\md\vpa \,,
\end{equation}
with the abbreviations
$$
   f' = f(\vp',t) = f(\Phi^{-s}(\vp,\vpa),t) \,,\qquad f_*' = f(\vp_*',t) = f(\Phi^{-s}_*(\vp,\vpa),t) \,,\qquad\sigma(s) = 2\lambda\gamma e^{-\gamma s} \,.
$$
The factor $e^s$ in the gain term is the determinant of the Jacobian of the non-volume-preserving collision rules (as in the dissipative Boltzmann equation, see e.g. 
\cite{c5:toscani}).

A weak formulation of the collision operator is derived by using the symmetry $\vp\leftrightarrow\vpa$ and by the transformation to post-collisional states in the gain term:
\begin{align*}
    \int_{\R} Q_1(f,f)h\:\md\vp = \frac{1}{2}\int_{\R^2}\int_0^\infty \sigma(s)f\fa \Bigl( & h\left( \Phi^s(\vp,\vpa)\right) 
    + h\left( \Phi_*^s(\vp,\vpa)\right)  - h(\vp) - h(\vpa) \Bigr) \md s\:\md\vpa\md\vp
\end{align*}
The choices $h(\vp) = 1$ and $h(\vp)=\vp$ show that the conservation laws of the non-instantaneous model remain valid:
$$
     M(t) := \int_{\R} f(\vp,t)\md\vp = M(0) \,,\qquad I(t) := \int_{\R} \vp f(\vp,t)\md\vp = I(0) \,.
$$
Finally we choose $h(\vp) = (\vp - \vp_\infty)^2$ with $\vp_\infty = \frac{I}{M}$ and obtain
\begin{eqnarray*}
   \int_{\R} Q_1(f,f)(\vp-\vp_\infty)^2 \md\vp &=& - \frac{1}{4} \int_0^\infty \sigma(s)(1-e^{-2s})\md s \int_{\R^2} f\fa(\vp-\vpa)^2 \md\vpa\md\vp \\
   &=& - \frac{2\lambda M}{2+\gamma} \int_{\R} (\vp-\vp_\infty)^2 f\:\md\vp \,,
\end{eqnarray*}
implying exponential decay of the variance:
$$
    V(t) := \int_{R} (\vp-\vp_\infty)^2 f(\vp,t)\md\vp = \exp\left( - \frac{2\lambda M}{2+\gamma} t\right) V(0) \,.
$$
As in the non-instantaneous case, the solution concentrates as $t\to\infty$.

\begin{theorem}\label{f-kin-ex-un-dec}
Let the initial data satisfy \eqref{IC-ass}. Then \eqref{f-kin} with $f(t=0)=f_0$ has a unique solution in $C([0,\infty); L^1(\R))$, satisfying
$$
    \lim_{t\to\infty} f(t) = M_{f_0}\delta_{\vp_\infty} \,,
$$
in the sense of distributions with $\vp_\infty = I_{f_0}/M_{f_0}$.
\end{theorem}

The proof is very similar to the proofs of Theorems \ref{thm:ex_un} and \ref{c5:decay} and therefore omitted.

\paragraph{The rigorous limit:}

We start by looking at the rescaled ODEs for the moments:
\begin{equation}\label{ODE-eps}
	\begin{split}
		\dot M_{\feps} &= 2\gamma M_{\geps} - 2\lambda M_{\feps}^2 \,,\qquad \ve\dot M_{\geps}  = \lambda M_{\feps}^2 - \gamma M_{\geps} \,,\\
		\dot I_{\feps} &= 2\gamma I_{\geps} - 2\lambda I_{\feps} M_{\feps} \,,\qquad \ve\dot I_{\geps} = \lambda I_{\feps} M_{\feps} - \gamma I_{\geps} \,,	\\
		\dot V_{\feps} &= 2\gamma V_{\geps} - 2\lambda M_{\feps} V_{\feps} \,,\qquad
		\ve\dot V_{\geps} = \lambda M_{\feps} V_{\feps} - \gamma V_{\geps} - \bar V_{\geps} \,,\\
		\ve\dot{\bar V}_{\geps} &= 2\lambda M_{\feps} V_{\feps} - (2+\gamma)\bar V_{\geps} + 2\lambda(I_{\feps} - \vp_{\infty,\ve} M_{\feps})^2 \,,
	\end{split}
\end{equation}
where 
$$
  \vp_{\infty,\ve} = \frac{I_{f_0} + 2\ve I_{g_0}}{M_{f_0} + 2\ve M_{g_0}} = \vp_\infty + O(\ve) \,,\qquad \text{with }\vp_\infty =  \frac{I_{f_0}}{M_{f_0}}\,,
$$
is used in the definition of the variances $V_{\feps}$ and $V_{\geps}$.

Standard results \cite{Fenichel} of the theory for singularly perturbed ODEs apply to these three systems. In the language of 
singular perturbation theory, the moments of $g$ are \emph{fast variables} and the moments of $f$ \emph{slow variables}. In an \emph{initial layer} of 
$O(\ve)$-length the slow variables remain approximately constant whereas the dynamics of the fast variables in terms of the \emph{layer variable} $\tau = t/\ve$
is approximately governed by the \emph{layer equations}
\begin{align*}
   &\frac{\md \hat M_g}{\md\tau} = \lambda M_{f_0}^2 - \gamma \hat M_g \,,\qquad
	\frac{\md \hat I_g}{\md\tau} = \lambda I_{f_0} M_{f_0} - \gamma \hat I_g \,,\\
   &\frac{\md \hat V_g}{\md\tau} = \lambda M_{f_0} V_{f_0} - \gamma \hat V_g - \hat{\bar V}_g \,,\qquad
	\frac{\md\hat{\bar V}_g}{\md\tau} = 2\lambda M_{f_0} V_{f_0} - (2+\gamma)\hat{\bar V}_g + 2\lambda(I_{f_0} - \vp_\infty M_{f_0})^2 \,.
\end{align*}
The important property of this system is its stability: As $\tau\to\infty$ the solution $(\hat M_g,\hat I_g,\hat V_g,\hat{\bar V}_g)$ converges exponentially 
to its steady state. Away from the initial layer the solution can be approximated by the solution of the \emph{reduced system}, obtained by setting $\ve=0$
in \eqref{ODE-eps}. After elimination of the fast variables by the algebraic equations, this becomes an ODE system for the moments of $f$, whose solutions
converge as $t\to\infty$. It is the main result of \cite{Fenichel} that these formal approximations are uniformly valid. In the context of the present work, we only need
a simple immediate consequence:

\begin{lemma}\label{moments-eps}
Let the initial data satisfy \eqref{IC-ass}. Then, for any $\ve_0>0$, the solution 
$$(M_{\feps},M_{\geps},I_{\feps},I_{\geps},V_{\feps},V_{\geps},\bar V_{\geps})(t)$$ 
of the initial value problem for \eqref{ODE-eps} is bounded uniformly in $t\ge 0$ and $0<\ve\le \ve_0$.
\end{lemma}

These uniform bounds for the moments are the essential prerequisite for the rigorous instantaneous limit.

\begin{theorem}
	Let the initial data satisfy \eqref{IC-ass}, $f_0\log f_0\in L^1(\R)$, and $g_0\log g_0 \in L^1(\R^2)$. Then the solution $(\feps,\geps)$ of \eqref{IC}, \eqref{model_eps} satisfies
	\begin{align}\label{c5:convergence}
	        \begin{split}
		&\lim_{\ve\to 0} \feps(\cdot,t) = f(\cdot,t) \quad\text{weakly in $L^1(\R)$, locally uniformly in }t\in [0,\infty) \,,\\
		&\lim_{\ve\to 0} \geps(\cdot,t) = g(\cdot,t) \quad\text{tightly, locally uniformly in }t\in (0,\infty) \,,
		\end{split}
	\end{align}
where $f$ is the solution of \eqref{f-kin} with $f(t=0)=f_0$ and $g$ is given by \eqref{mild-limit}. 
\end{theorem}
\begin{proof}
We recall the definition \eqref{entr} and the time derivative \eqref{entr-diss} of the entropy. For the rescaled problem we obtain
$$
    \frac{\md}{\md t} \left(\int_{\R} \feps(\log(\lambda \feps) - 1)\md\vp + \ve\int_{\R^2} \geps \left( \log(\gamma \geps) - 1 \right)\md\vpa\md\vp \right) \le M_{\geps} \,,
$$
which, by Lemma \ref{moments-eps}, implies boundedness of $\feps\log\feps$ in $L^1(\R)$ on bounded time intervals uniformly in $\ve$.

Due to the boundedness of the masses as well as the variances we see that for any bounded time interval $\{f_\ve\}_\ve$ and $\{g_\ve\}_\ve$ are \emph{tight sets} of measures. Due to the \emph{Prokhorov theorem} \cite{c5:P} this is equivalent to weak sequential compactness of $\{\feps\}_\ve$ and $\{\geps\}_\ve$ in the space of measures. For 
$\{\feps\}_\ve$ this can be improved by the entropy bound and the Dunford-Pettis theorem to weak sequential compactness in $L^1(\R\times (0,T))$ for any $T>0$.

A further improvement is the consequence of the estimate 
\begin{equation}\label{dtf-est}
    \|\pa_t \feps\|_{L^1(\R)} \le 2\gamma M_{\geps} + 2\lambda M_{\feps}^2 \,,
\end{equation}
implying, again with Lemma \ref{moments-eps}, uniform Lipschitz continuity of the map $t \mapsto \feps(\cdot,t)$ with respect to the $L^1(\R)$-topology. 
As a consequence there exists $f\in C([0,\infty); L^1(\R))$ such that a sequence $\{f_{\ve_n}\}$, with $\ve_n\to 0$,  converges to $f$ locally uniformly in $t\in [0,\infty)$ with respect 
to the weak topology in $L^1(\R)$.

The same cannot be expected for the fast variable $\geps$, where we also lack the information from the entropy. However, we consider the mild formulation
\begin{equation}\label{mild-eps}
  \geps(\vp,\vpa,t) = (S(t/\ve)g_0)(\vp,\vpa) + \frac{\lambda}{\ve} \int_0^t \left(S(s/\ve)\feps(\cdot,t-s) f_{\ve,\ast}(\cdot,t-s)\right)(\vp,\vpa) \: \md s \,,
\end{equation}
and, with a test function $h\in C_b(\R^2)$ and with $t_1\ge t_2\ge \underline t > 0$, use it in
\begin{align*}
   &\int_{\R^2} h(\vp,\vpa)(\geps(\vp,\vpa,t_1) - \geps(\vp,\vpa,t_2))\md\vpa\md\vp \\
   &= \int_{\R^2} \left(e^{-\gamma t_1/\ve}h(\Phi^{t_1/\ve}, \Phi_*^{t_1/\ve}) - e^{-\gamma t_2/\ve}h(\Phi^{t_2/\ve}, \Phi_*^{t_2/\ve})\right)g_0 \:\md\vpa\md\vp \\
   &\quad + \frac{\lambda}{\ve} \int_{t_2}^{t_1} e^{-\gamma s/\ve} \int_{\R^2} h(\Phi^{s/\ve}, \Phi_*^{s/\ve}) \feps(\cdot,t_1-s)f_{\ve,\ast}(\cdot,t_1-s)
   \md\vpa\md\vp\:\md s \\
    &\quad + \frac{\lambda}{\ve} \int_0^{t_2} e^{-\gamma s/\ve} \int_{\R^2} h(\Phi^{s/\ve}, \Phi_*^{s/\ve}) \left(\feps(\cdot,t_1-s)f_{\ve,\ast}(\cdot,t_1-s)
   - \feps(\cdot,t_2-s) f_{\ve,\ast}(\cdot,t_2-s) \right)\md\vpa\md\vp\:\md s  \\
   &=: I_1 + I_2 + I_3 \,.
 \end{align*}
 The three terms are estimated separately. Let $\psi$ be a continuity modulus of $h$, such that $|h(\vp,\vpa)-h(\tilde\vp,\tilde\vpa)|\le \psi(|\vp-\tilde\vp| + |\vpa-\tilde\vpa|)$.
 It can be chosen nondecreasing, continuous, and (by the boundedness of $h$) bounded.
 \begin{align*}
    |I_1| &\le \frac{t_1-t_2}{\ve}e^{-\gamma\underline t/\ve} \gamma \|h\|_{L^\infty(\R^2)}M_{g_0} 
      + e^{-\gamma\underline t/\ve} \int_{\R^2}\psi\left(\frac{t_1-t_2}{\ve}e^{-\underline t/\ve}|\vp-\vpa| \right)g_0\:\md\vpa\md\vp \\
     &\le \frac{t_1-t_2}{\underline t \,e} \|h\|_{L^\infty(\R^2)}M_{g_0} 
      + \int_{\R^2}\psi\left(\frac{t_1-t_2}{\underline t\,e}|\vp-\vpa| \right)g_0\:\md\vpa\md\vp \,,
\end{align*}
where we have used $ze^{-z}\le e^{-1}$. The right hand side is independent of $\ve$ and tends to zero as $t_1-t_2\to 0$, where for the second term dominated 
convergence can be employed.
\begin{align*}
  |I_2| &\le \frac{\lambda}{\gamma} \left( e^{-\gamma t_2/\ve} - e^{-\gamma t_1/\ve}\right) \|h\|_{L^\infty(\R^2)} M_{\feps}^2
    \le \frac{\lambda(t_1 - t_2)}{\ve} e^{-\gamma \underline t/\ve} \|h\|_{L^\infty(\R^2)} M^2 \\
    &\le \frac{\lambda(t_1 - t_2)}{\gamma\underline t\,e}  \|h\|_{L^\infty(\R^2)} M^2 \,,
\end{align*}
with the same result as for $I_1$.
\begin{align*}
  |I_3| &\le \frac{\lambda}{\gamma} \|h\|_{L^\infty(\R^2)} 2M_{\feps} (t_1-t_2) \sup_{s>0} \|\pa_t \feps\|_{L^1(\R)} \,.
\end{align*}
With \eqref{dtf-est} and with the uniform boundedness of the moments, these three estimates together imply tight equicontinuity of $\{\geps\}_\ve$ with respect to 
$t\in (\underline t,\infty)$. This implies that  there exists a measure $g$ such that a sequence $\{g_{\ve_n}\}$, where w.l.o.g. $\{\ve_n\}$ is the same as above, 
converges to $g$ tightly and uniformly in $t\in (\underline t,\infty)$ for every $\underline t > 0$.

The weak formulation of the mild formulation \eqref{mild-eps} of the $\geps$-equation at time $t$ can be written as 
\begin{align*}
   \int_{\R^2} h\geps \md\vpa\md\vp =& e^{-\gamma t/\ve} \int_{\R^2} h(\Phi^{t/\ve}, \Phi_*^{t/\ve})g_0 \md\vpa\md\vp \\
   &+ \lambda \int_0^{t/\ve} e^{-\gamma \sigma} \int_{\R^2} h(\Phi^\sigma, \Phi_*^\sigma) \feps(\cdot,t-\ve \sigma) f_{\ve,\ast}(\cdot,t-\ve\sigma)\md\vpa\md\vp\:\md\sigma
\end{align*}
If $h\in C_b(\R^2)$ then for fixed $\sigma$ also $h(\Phi^\sigma, \Phi_*^\sigma)\in C_b(\R^2)$. Since weak convergence of two measures implies weak convergence
of the product measure to the product measure of the limits \cite[Theorem 2.8 (ii)]{Billingsley}, we can pass to the limit $\ve\to 0$ in the last integral over $\R^2$. 
Passage to the limit in the integral with respect to $\sigma$ is then a consequence of dominated convergence. The first term on the right hand side obviously tends to zero 
for every $t>0$. This shows that the limits $f$ and $g$ satisfy \eqref{mild-limit}. Passing to the limit in the distributional formulation of the $\feps$-equation is straightforward
since $M_{\feps}\to M_f = M_{f_0}$ by weak convergence of $\feps$.

Finally the restriction to subsequences is not necessary by the uniqueness result in Theorem \ref{f-kin-ex-un-dec}.
\end{proof}

\section{Formal properties of the Deterministic Collision Time Model} \label{c5:noninstdet} 

Let $D:= \{(\vp,\vp):\, \vp\in\R\}$ denote the main diagonal in the $(\vp,\vpa)$-plane with length element $\md\ell= \sqrt{2} \:\md\vp$ and let $\nu_{\pm} = \pm 2^{-1/2}(1,-1)$ 
be the unit outward normal vector along $D$ for the domain $\vp \lessgtr \vpa$. The computation
\begin{align*}
  &\int_{\R^2} \nabla\cdot (v_2 g) \md\vpa\md\vp = \int_{\vp<\vpa} \nabla\cdot (v_2 g) \md\vpa\md\vp + \int_{\vp>\vpa} \nabla\cdot (v_2 g) \md\vpa\md\vp \\
  &= \int_{D} \left( \nu_+\cdot (v_2 g)(\vp,\vp+) + \nu_-\cdot (v_2 g)(\vp,\vp-)\right) \md\ell = 2\int_{\R} \bar g\:\md\vp
\end{align*}
justifies the choice of the source term in the first equation of \eqref{mainmodel-d}, since it implies mass conservation:
\begin{equation}\label{masscons-d}
    \frac{\md}{\md t} (M_f + 2M_g) = 0 \,,\qquad\text{i.e.,}\quad M_f(t) + 2 M_g(t) = M := M_{f_0} + 2 M_{g_0}\,.
\end{equation}
Similarly to above we compute
\begin{align*}
  &\int_{\R^2} \vp\nabla\cdot (v_2 g) \md\vpa\md\vp = 2\int_{\R} \vp\bar g\:\md\vp - \frac{1}{2} \int_{\R^2} g\sgn(\vpa - \vp)\md\vpa\md\vp \,,
\end{align*}
where the last integral vanishes by oddness of the integrand. Therefore we have, as for the SCTM, also the second conservation law
$$
    \frac{\md}{\md t} (I_f + 2I_g) = 0 \,,\qquad\text{i.e.,}\quad I_f(t) + 2 I_g(t) = I := I_{f_0} + 2 I_{g_0}\,.
$$
An important difference to the SCTM is the lack of complete information on the dynamics of $M_f,M_g,I_f,I_g$. The rate of particles leaving the collision state
is given in terms of the trace $\bar g$ and cannot be expressed in terms of the moments. With the mean
\begin{align}\label{c5:deterministic_mean}
	\vp_{\infty} := \frac{I}{M} \,,
\end{align}
we define the variances $V_f$, $V_g$ and obtain, again with a similar computation,
\begin{align*}
	\frac{\md}{\md t} (V_f + 2V_g) =  - 2 \int_{\R^2} |\vp-\vpa| g \, \md \vpa \md \vp \,,
\end{align*}
i.e. the variance is nonincreasing as for the SCTM. However, we do not get any additional information. Therefore we do not have a rigorous constructive result concerning
decay to equilibrium. Formally, from the dissipation term above we expect $g$ to concentrate along the diagonal as $t\to\infty$. Therefore we do not expect any more collision
dynamics after long time, which implies that also the source term $\lambda f\fa$ should concentrate along the diagonal. For the tensor product $f\fa$ this is only possible if
$f$ is concentrated at one point. From the right hand side of the $f$-equation we then deduce that also the trace $\bar g$ concentrates and that the limiting masses
satisfy $2M_{g_\infty} = \lambda M_{f_\infty}^2$. Thus, we expect convergence to the equilibrium state
\begin{align}
	(f_{\infty}(\vp),g_{\infty}(\vp,\vpa)):= \left(M_{f_{\infty}} \delta(\vp-\vp_\infty), \frac{\lambda M_{f_{\infty}}^2}{2} \delta(\vpa-\vp)\delta(\vp-\vp_\infty)\right) \,,
\end{align}
with 
$$
   M_{f_\infty} = \frac{2M}{1+\sqrt{1+4\lambda M}} \,.
$$

\section{Existence and uniqueness for the Deterministic Collision Time Model}\label{sec:det_ex}

We start with the mild formulation of the $g$-equation in \eqref{mainmodel-d}:
\begin{equation}\label{mild-g}
   g(\cdot,\cdot,t) = g_0(\Phi^{-t},\Phi_*^{-t}) + \lambda \int_0^t f(\Phi^{s-t},s) f(\Phi_*^{s-t},s) \md s \,,
\end{equation}
with 
$$
  \Phi^{-t}(\vp,\vpa) = \vp + \frac{t}{2}\sgn(\vp-\vpa) \,,\qquad \Phi_*^{-t}(\vp,\vpa) = \vpa - \frac{t}{2}\sgn(\vp-\vpa) \,.
$$
By the indistinguishability property, the trace of $g$ along the diagonal can be written as
$$
   \bar g(\vp,t) = g_0\left(\vp + \frac{t}{2},\vp - \frac{t}{2}\right) + \lambda \int_0^t f\left(\vp + \frac{t-s}{2},s\right)f\left(\vp-\frac{t-s}{2},s\right)\md s \,.
$$
Note that for $g_0\in L^1(\R^2)$ this is in general not in $L^1(\R)$ for fixed $t$. This requires some care in the formulation of the problem, which we
write in a mild formulation for $f$, eliminating $g$ by substitution of the above:
\begin{align}
    f(\vp,t) = &F_f(0,t)f_0(\vp) + 4 \int_0^t F_f(s,t) g_0\left(\vp+\frac{s}{2},\vp-\frac{s}{2}\right)\md s \nonumber\\
    &+ 4\lambda \int_0^t F_f(s,t) \int_0^s f\left(\vp+\frac{s-r}{2},r\right)f\left(\vp-\frac{s-r}{2},r\right)\md r \:\md s \,,\label{mild-f}
\end{align}
with
$$
    F_f(s,t) = \exp\left(-2\lambda \int_s^t M_f(r)\md r\right) \,.
$$

\begin{theorem}\label{existence}
Let $(f_0,g_0) \in L_+^1(\R)\times L_+^1(\R^2)$. Then \eqref{IC}, \eqref{mainmodel-d} has a unique mild (in the sense \eqref{mild-g}, \eqref{mild-f}) solution 
$(f,g) \in L^\infty\big((0,\infty);\: L_+^1(\R)\times L_+^1(\R^2)\big)$.
\end{theorem}
\begin{proof}
We denote the right hand side of \eqref{mild-f} by $\mathcal{F}[f](\vp,t)$ and note that the fixed point map $\mathcal{F}$ preserves nonnegativity and, by
$F_f(s,t)\le 1$ and the consequence
\begin{eqnarray*}
    M_{\mathcal{F}[f]}(t) &\le& M_{f_0} + 4 M_{g_0} + 4\lambda \int_0^t \left(\int_r^t \int_{\R} f\left(\vp+\frac{s-r}{2},r\right)f\left(\vp-\frac{s-r}{2},r\right)\md\vp\:\md s\right) \md r\\
    &\le& M_{f_0} + 4 M_{g_0} + 4\lambda \int_0^t M_f(r)^2\md r \,,
\end{eqnarray*}
it maps $L^\infty\big((0,T);\: L_+^1(\R)\big)$ into itself. Here we have used the coordinate transformation
\begin{equation}\label{phis2psi}
   (\vp,s) \to (\psi,\psi_*) = \left( \vp+\frac{s-r}{2}, \vp-\frac{s-r}{2}\right) \,.
\end{equation}
More precisely, with $\bar M := 2 M_{f_0} + 8M_{g_0}$, the set
$$
   \mathbb{F} := \left\{ f\in L^\infty\big((0,T];\: L_+^1(\R)\big):\, \sup_{0<t<T}M_f(t) \le \bar M\right\}
$$
is mapped into itself, if $T$ is small enough such that $8\lambda T \bar M \le 1$.

For proving a contraction property, we choose
$f$, $\tf \in \mathbb{F}$ and estimate
\begin{eqnarray*}
	\|\mathcal{F}[f](\cdot,t)-\mathcal{F}[\tf](\cdot,t)\|_{L^1(\R)} &\le&  4\lambda \int_0^t \int_r^t \int_{\R^2} \Biggl| F_f(s,t) f\left(\vp+\frac{s-r}{2},r\right)f\left(\vp-\frac{s-r}{2},r\right) \\
	&&\qquad\quad - F_{\tf}(s,t) \tf\left(\vp+\frac{s-r}{2},r\right)\tf\left(\vp-\frac{s-r}{2},r\right) \Biggr|  \md \vp \,\md s\, \md r \\
	&\le& I + II + III\,,
\end{eqnarray*}
where the splitting into three terms comes from estimating the integrand by
$$
   |F_f f\fa - F_{\tf}\tf\tf_*| \le |F_f - F_{\tf}| f\fa + |f - \tf| \fa + \tf |\fa -\tf_*| \,.
$$
For the first term we use
$$
   |F_f(s,t) - F_{\tf}(s,t)| \le 2\lambda \int_s^t |M_f(r) - M_{\tf}(r)| \md r \le 2\lambda T \|f - \tf\|_{L^\infty((0,T);\: L_+^1(\R))}
$$
and the transformation \eqref{phis2psi} to obtain
$$
    I \le 8(\lambda T \bar M)^2 \|f - \tf\|_{L^\infty((0,T);\: L_+^1(\R))} \,.
$$
For the second and the third term the coordinate change \eqref{phis2psi} immediately gives
$$
   II + III \le 8\lambda T \bar M \|f - \tf\|_{L^\infty((0,T);\: L_+^1(\R))}  \,.
$$
By these estimates $\mathcal{F}$ is a contraction on $\mathbb{F}$ for $T$ small enough, implying local existence. Global existence is then a consequence 
of the mass conservation property \eqref{masscons-d}, implying $M_f(t) \le M_{f_0} + 2 M_{g_0}$.

Finally it is straightforward to check that $g$ defined by \eqref{mild-g} satisfies the properties stated in the theorem.
\end{proof}

\section{The formal instantaneous limit for the Deterministic Collision Time Model}\label{sec:il2}

With the same motivation as in Section \ref{sec:il1} we introduce the rescaling
$$
v_2 \to \frac{v_2}{\ve} \,, \qquad g \to \ve g \,,
$$
in \eqref{mainmodel-d}: 
\begin{equation}\label{c5:deterministic_modeleps}
	\begin{split}
		&\pa_t \feps = 2\left(2 \bar{g}_\ve - \lambda M_{\feps} \feps \right) \,, \\
		&\ve \pa_t \geps + \nabla \cdot (v_2 \geps) = \lambda \feps f_{\ve,\ast} \,.
	\end{split}
\end{equation}
Solving the initial value problem for the second equation leads to the diagonal trace
$$
   \bar g_\ve(\vp,t) = g_0\left(\vp + \frac{t}{2\ve},\vp - \frac{t}{2\ve}\right) + \lambda \int_0^{t/\ve} \feps\left(\vp + \frac{s}{2},t-\ve s\right)\feps\left(\vp-\frac{s}{2},t-\ve s\right)\md s \,,
$$
with the formal limit
\begin{eqnarray*}
   \bar g(\vp,t) &=&  \lambda \int_0^\infty f\left(\vp + \frac{s}{2},t\right) f\left(\vp-\frac{s}{2},t\right)\md s 
   =  \frac{\lambda}{2} \int_{\R} f\left(\vp + \frac{s}{2},t\right) f\left(\vp-\frac{s}{2},t\right)\md s \\
   &=& \lambda \int_{\R} f(\vp',t)f(\vpa,t)\md\vpa \,,\qquad\text{with } \vp' = 2\vp-\vpa \,,
\end{eqnarray*}
assuming that $g_0$ decays at infinity. The limiting kinetic equation for $f$ can therefore be written as 
\begin{equation}\label{f-kin2}
    \pa_t f = Q_2(f,f) := 2\lambda \int_{\R} (2f'\fa - f \fa)\md\vpa \,.
\end{equation}
This is a \emph{sticky particle model,} where particles with pre-collisional states $\vp', \vpa$ have the same post-collisional state $\vp = \frac{\vp' + \vpa}{2}$.

We observe that the model obtained after performing the instantaneous limit corresponds to the usual midpoint/alignment-model, recent matter of investigation in \cite{c5:DFR, c5:HKMS} and with additional noise term in \cite{c5:carlen}.

The weak formulation of $Q_2$ is obtained by transformation to pre-collisional states in the gain term and by symmetrization in the loss term:
\begin{align}\label{c5:Qd_weak}
	&\int_{\R} Q_2(f,f) h \, \md \vp  =2\lambda  \int_{\R^2} f \af \left(h\left(\frac{\vp+\vpa}{2}\right)-\frac{h(\vp)+h(\vpa)}{2}\right) \, \md \vpa \md \vp \,.
\end{align}
With $h(\vp) = 1$ and $h(\vp)=\vp$ we observe that the conservation laws are preserved in the limit:
\begin{align}\label{c5:deterministic_inst_mass}
	M_f(t) = M_{f_0} \,,\qquad I_f(t) = I_{f_0} =: M_{f_0}\vp_\infty \,.
\end{align} 
With $h(\vp)=(\vp-\vp_{\infty})^2$ we note that the variance 
\begin{align*}
	V_f := \int_{\R} (\vp-\vp_\infty)^2 f \, \md \vp.
\end{align*}
satisfies the ODE
\begin{align*}
	\frac{\md}{\md t} V_f = -\frac{\lambda}{2} \int_{\R^2}(\vp-\vpa)^2 f f_* \, \md \vpa \md \vp  = -\lambda M_{f_0} V_f \,,
\end{align*}
and therefore decays exponentially.

Finally we note that the existence, uniqueness and decay-to-equilibrium result Theorem 4 also holds for \eqref{f-kin2}. Again the proof is rather straightforward and
omitted.

\end{document}